\documentclass{IEEEtran}
\usepackage{amsmath,amssymb,amsfonts}
\usepackage{algorithm,algorithmic}
\usepackage{graphicx}
\usepackage{textcomp}

\usepackage{color}

\usepackage{lipsum}
\usepackage{epstopdf}

\usepackage{amsopn}

\usepackage{amsthm}

\usepackage{enumerate}
\usepackage{cancel}
\usepackage{mathrsfs}
\usepackage{mathdots}
\usepackage{euscript}
\usepackage{amscd}
\usepackage{cite}
\usepackage{placeins}
\usepackage{tikz}
\usetikzlibrary{snakes,arrows,shapes}

\graphicspath{{images/}}

\newtheorem{theorem}{Theorem}
\newtheorem{corollary}{Corollary}
\newtheorem{lemma}{Lemma}
\newtheorem{proposition}{Proposition}
\theoremstyle{definition}
\newtheorem{example}{Example}
\theoremstyle{definition}\newtheorem{problem}{Problem}
\theoremstyle{remark}\newtheorem{remark}{Remark}

\newcommand{\bmat}{\left[ \begin{matrix}}
	\newcommand{\emat}{\end{matrix} \right]}

\DeclareMathOperator{\trace}{\rm tr}

\newcommand{\E}{{\mathbb E}}
\newcommand{\Rbb}{\mathbb R}

\newcommand{\Jbb}{\mathbb J}
\newcommand{\Cbb}{\mathbb C}

\newcommand{\Tbb}{\mathbb T}

\newcommand{\oneb}{\mathbf 1}

\newcommand{\Cb}{\mathbf C}

\newcommand{\Vb}{\mathbf V}
\newcommand{\Yb}{\mathbf Y}

\newcommand{\Lambdab}{\boldsymbol{\Lambda}}

\newcommand{\zerob}{\boldsymbol{0}}

\DeclareMathOperator{\range}{Range}

\newcommand{\Hfrak}{\mathfrak{H}}

\newcommand{\Cfrak}{\mathfrak{C}}

\newcommand{\Sfrak}{\mathfrak{S}}
\newcommand{\Sscr}{\mathscr{S}}
\newcommand{\Lscr}{\mathscr{L}}
\newcommand{\Cscr}{\mathscr{C}}

\newcommand{\Bcal}{\mathcal{B}}

\def\BibTeX{{\rm B\kern-.05em{\sc i\kern-.025em b}\kern-.08em
    T\kern-.1667em\lower.7ex\hbox{E}\kern-.125emX}}
    
\begin{document}
\title{On the Well-Posedness of a Parametric Spectral Estimation Problem and Its Numerical Solution}
\author{Bin Zhu
\thanks{Submitted for review April 10, 2018. This work was funded by the China Scholarship Council (CSC) under file no.~201506230140.}
\thanks{B. Zhu is with the Department of Information Engineering, University of Padova, Via Giovanni Gradenigo, 6b, 35131 Padova, PD, Italy (email: \texttt{zhubin@dei.unipd.it}).}
}

\maketitle

\begin{abstract}
This paper concerns a spectral estimation problem in which we want to find a spectral density function that is consistent with estimated second-order statistics. It is an inverse problem admitting multiple solutions, and selection of a solution can be based on prior functions. We show that the problem is well-posed when formulated in a parametric fashion, and that the solution parameter depends continuously on the prior function. In this way, we are able to obtain a smooth parametrization of admissible spectral densities. Based on this result, the problem is reparametrized via a bijective change of variables out of a numerical consideration, and then a continuation method is used to compute the unique solution parameter. Numerical aspects such as convergence of the proposed algorithm and certain computational procedures are addressed. A simple example is provided to show the effectiveness of the algorithm.
\end{abstract}


\section{Introduction}\label{sec:intro}

Spectral estimation of a random signal is an important problem in modeling and identification. The current framework of the problem, called ``THREE'', was pioneered in \cite{BGL-THREE-00} in the scalar case and further developed in \cite{Georgiou-L-03,Georgiou-02,georgiou2002spectral}. A brief description of the procedure goes as follows. In order to estimate the unknown spectrum of a stationary process, one first feeds it into a bank of filters and collects the steady output covariance matrix as data. Then the problem is to find a spectral density that is consistent with the covariance data, which naturally admits a formulation of moment-like equations. We shall also call this a \emph{generalized moment problem} as it will be the central object of investigation of this paper.

The idea of formulation as a moment problem can be traced back to the (scalar) \emph{rational covariance extension problem} formulated by Kalman \cite{Kalman}, which aims to find an infinite extension of a finite covariance sequence such that the resulting spectral density, i.e., the Fourier transform of the infinite sequence, is a rational function. The problem was solved first partially by Georgiou \cite{Georgiou-87,Gthesis}, and then completely after a series of works \cite{BLL97,BGL98,BLGM-95,SIGEST-01,byrnes1997partial} by Byrnes, Lindquist, and coworkers. The problem of covariance extension is also closely connected to the analytic interpolation problem of various generality \cite{BGL-01,BLN-03,BTLM06,Georgiou-87-NP,georgiou1999interpolation,takyar2010analytic}. These theories have a wide range of applications in the fields of systems and control, circuit theory, and signal processing (cf.~the afore cited papers and references therein).

Similar to its classical counterpart \cite{KreinNudelman,A65moments}, the generalized moment problem has infinitely many solutions when a solution exists, except for certain degenerate cases. Therefore, such a problem is not well-posed in the sense of Hadamard\footnote{Recall that a problem is well-posed if 1) a solution exists; 2) the solution is unique; 3) the solution depends continuously on the data.}. The mainstream approach today to promote uniqueness of the solution is built on calculus of variations and optimization theory. It has two main ingredients. One is the introduction of a prior spectral density function $\Psi$ as additional data, which represents our ``guess'' of the desired solution $\Phi$. The other is a cost functional $d(\cdot,\cdot)$, which is usually a distance (divergence) between two spectral densities. Then one tries to solve the optimization problem of minimizing $d(\Phi,\Psi)$ subject to the (generalized) moment equation as a constraint. We mention some works that explore different cost functionals, including the Kullback-Leibler divergence \cite{Georgiou-L-03,avventi2011spectral,PavonF-06,FPR-07,FRT-11,baggio2018further} and its matricial extension \cite{Georgiou-06}, the Hellinger distance \cite{FPR-08,RFP-09,RFP-10-well-posedness}, the Itakura-Saito distance \cite{FMP-12,GL-17,enqvist2008minimal}, and some more general divergence families \cite{Z14rat,Z14,Z15}.

The optimization approach guarantees existence and uniqueness of the solution (identifiability) as cited above. However, continuous dependence of the solution on the data does not seem to have attracted much attention, especially in the multivariate case. For the scalar rational covariance extension problem \cite{BLGM-95,BLL97,BGL98,SIGEST-01,byrnes2002identifiability}, such continuity argument is actually part of the results of well-posedness. More precisely, in \cite{BLGM-95} the correspondence between the covariance data and the solution vector has been shown to be a diffeomorphism, i.e., a $C^1$ function with a $C^1$ inverse. Not many results in this respect exist in multivariate formulations. We mention \cite{RFP-10-well-posedness,zhu2017parametric}, where continuous dependence of the solution on the covariance matrix has been shown in the contexts of optimization with the Hellinger distance and a certain parametric formulation, respectively.

The present work can be seen as a continuation of \cite{zhu2017parametric}. As one main contribution, we shall here show that, when restricted to a predefined family of spectral densities, the unique solution parameter to the spectral estimation problem depends also continuously on the prior function under a suitable metric topology. The idea is to study the so-called moment map directly in a parametric form. Due to the regularity of the moment map, we can view the solution parameter as an implicit functional of the prior, and then invoke the Banach space version of the implicit function theorem to prove continuity. Based on this continuity result, the rest of the paper is devoted to a numerical solver to compute the solution (parameter) using a continuation method which is a quite standard tool from nonlinear analysis. The idea is to solve a family of moment equations parametrized by one real variable, and to trace the solution curve from a known starting point. An specialized algorithm called ``predictor-corrector'' (see \cite{allgower1990continuation}) is adapted for the current problem and a conservative bound on the step length of the algorithm is given to ensure convergence. The proof, inspired by \cite{enqvist2001homotopy}, is built upon the Kantorovich theorem for the convergence of Newton iterations to solve nonlinear equations. Moreover, we do computation in the domain of spectral factors due to the improvement of conditioning, especially when the solution lies near the boundary of the feasible set.


The paper is organized as follows. Problem formulation is given in Section \ref{sec:problem} where a parametric family of candidate solutions is also  introduced. In Section
\ref{sec:well-posed}, we first give a continuity argument with respect to the prior function, thus extending existing results on the well-posedness. Then we reformulate the problem in terms of the spectral factor of a certain rational spectral density without affecting the well-posedness. Section \ref{sec:numeric} contains a numerical procedure to compute the solution in the domain of the spectral factor using a continuation method. Convergence of the proposed algorithm is investigated in detail. Moreover, a key computational step concerning the inverse Jacobian is elaborated, and a numerical example is provided for illustration.

\subsection*{Notations}

Some notations are common as $\E$ denotes mathematical expectation, $\Cbb$ the complex plane, and $\Tbb$ the unit circle $\{\,z:|z|=1\,\}$. 

Sets: The symbol $\Hfrak_{n}$ represents the vector space of $n\times n$ Hermitian matrices, and $\Hfrak_{+,n}$ is the subset that contains positive definite matrices. The space of $\Hfrak_{m}$-valued continuous functions on $\Tbb$ is denoted with $C(\Tbb;\Hfrak_m)$. The set $C_+(\Tbb)$ consists of continuous functions on $\Tbb$ that take real and positive values, which is an open subset (under the metric topology) of $C(\Tbb)\equiv C(\Tbb;\Hfrak_1)$. The symbol $\Sfrak_m$ denotes the family of $\Hfrak_{+,m}$-valued functions defined on $\Tbb$ that are bounded and coercive. 

Linear algebra: The notation {$(\cdot)^{*}$ means taking complex conjugate transpose when applied to a matrix and $(\cdot)^{-*}$ is a shorthand for $[(\cdot)^{-1}]^{*}$. When considering a rational matrix-valued function with a state-space realization $G(z)=C(zI-A)^{-1}B+D$,  $G^{*}(z):=B^*(z^{-1}I-A^*)^{-1}C^*+D^*$.
Matrix inner product is defined as
$\langle A,B\rangle:=\trace(AB^*)$ for $A,B\in\Cbb^{m\times n}$, and $\|A\|_F:=\sqrt{\langle A,A\rangle}$ is the Frobenius norm.
The Euclidean $2$-norm of $x\in\Cbb^n$ is $\|x\|_2:=\sqrt{x^* x}$. The subscript $_2$ is usually omitted and we simply write $\|\cdot\|$.
When applied to a matrix $A\in\Cbb^{m\times n}$ or more generally a multilinear function, $\|A\|$ means the induced $2$-norm.

\section{Problem formulation}
\label{sec:problem}

Let us describe here the procedure to estimate the unknown spectral density $\Phi(z)$ of a zero-mean wide-sense stationary discrete-time $\Cbb^m$-valued process $y(t)$ in more details as proposed in \cite{Georgiou-L-03}. First we feed the process into a linear filter with a state-space representation
\begin{equation}\label{filter_bank}
x(t+1)=Ax(t)+By(t),
\end{equation}
whose transfer function is simply
\begin{equation}\label{trans_func}
G(z)=(zI-A)^{-1}B.
\end{equation}
There are some extra conditions on the system matrices. More precisely, $A\in\Cbb^{n\times n}$ is Schur stable, i.e., has all its eigenvalues strictly inside the unit circle and $B\in\Cbb^{n\times m}$ is of full column rank ($n\geq m$). Moreover, the pair $(A,B)$ is required to be \emph{reachable}.

Next an estimate of the steady-state covariance matrix $\Sigma:=\E\{x(t)x(t)^*\}$ of the state vector $x(t)$ is computed. Such structured covariance estimation problem has been discussed in the literature (see \cite{Zorzi-F-12, FPZ-12,ning2013geometry}). Here we shall assume that the matrix $\Sigma>0$ is given and we have
\begin{equation}\label{mmt_constraint}
\int G\Phi G^*=\Sigma.
\end{equation}
The integration is carried out on the unit circle $\Tbb$ with respect to the normalized Lebesgue measure $\frac{d\theta}{2\pi}$. This simplified notation will be adopted throughout the paper.

In general, an estimated covariance matrix may not be compatible with the filter structure (\ref{trans_func}). In other words, viewing (\ref{mmt_constraint}) as a constraint on the input spectrum, there may not exist a \emph{feasible} $\Phi$. In this paper we shall always assume such feasibility. Specifically, let us define the linear operator
\begin{equation}\label{map_Gamma}
\begin{split}
\Gamma\colon C(\Tbb;\Hfrak_m) & \to\Hfrak_n \\
 \Phi & \mapsto\int G\Phi G^*.
\end{split}
\end{equation}
Then we assume that the covariance matrix $\Sigma\in\range\Gamma$.
Equivalent conditions are elaborated in \cite{Georgiou-02,georgiou2002spectral} (see also \cite{FPZ-10,FPZ-12,Zorzi-F-12,FPR-07,FPR-08,RFP-09,RFP-10-well-posedness,FMP-12}).

Our problem now is to find a spectral density $\Phi$ that satisfies the generalized moment constraint (\ref{mmt_constraint}). It can been seen as a generalization of the classical \emph{covariance extension problem} \cite{Grenander_Szego}, in which we are given $p+1$ covariance matrices (with a slight abuse of notation) $\Sigma_k:=\E\{y(t+k)y(t)^*\}\in\Cbb^{m\times m},\ k=0,1,\dots,p$ such that the block-Toeplitz matrix
\begin{equation}
\Sigma=\bmat \Sigma_0&\Sigma_1^*&\Sigma_2^*&\cdots&\Sigma_p^* \\
\Sigma_1&\Sigma_0&\Sigma_1^*&\cdots&\Sigma_{p-1}^* \\
\Sigma_2&\Sigma_1&\Sigma_0&\cdots&\Sigma_{p-2}^* \\
\vdots&\vdots&\ddots&\ddots&\vdots \\
\Sigma_p&\Sigma_{p-1}&\cdots&\Sigma_1&\Sigma_0\emat
\end{equation}
is positive definite, and we want to find a spectral density $\Phi$ satisfying the set of \emph{moment equations}
\begin{equation}\label{mmt_eqns}
\int z^k\Phi=\Sigma_k,\quad k=0,1,\dots,p.
\end{equation}
It is easy to verify that these moment equations are equivalent to \eqref{mmt_constraint} with a choice of the matrix pair
\begin{equation}\label{A_B_cov_ext}
A=\bmat 0&I_m&0&\cdots&0 \\ 
0&0&I_m&\cdots&0 \\
\vdots&\vdots& &\ddots&\vdots \\
0&0&0&\cdots&I_m \\
0&0&0&\cdots&0\emat, \quad
B=\bmat 0 \\ 0 \\ \vdots \\ 0 \\ I_m\emat,
\end{equation}
and the transfer function
\begin{equation}\label{trans_func_covext}
G(z)=(zI-A)^{-1}B=\bmat z^{-p-1}I_m \\ z^{-p}I_m \\ \vdots \\z^{-1}I_m \emat.
\end{equation}
Here each block in the matrices is of $m\times m$ and $n=m(p+1)$. 

As mentioned in the Introduction and well studied in the literature, the inverse problem of finding $\Phi$ such that (\ref{mmt_constraint}) holds is typically not well-posed when feasible, because there are infinitely many solutions. One way to remedy this is to restrict the candidate solution to some particular family of spectral densities, as we shall proceed below.

Let us define the set of parameters
\begin{equation}
\Lscr_+:=\{\Lambda\in\Hfrak_{n}\;:\;G^*(z)\Lambda G(z)>0,\ \forall z\in\Tbb\}.
\end{equation}
By the continuous dependence of eigenvalues on the matrix entries, one can verify that $\Lscr_+$ is an open subset of $\Hfrak_{n}$. To avoid any redundancy in the parametrization, we have to define the set $\Lscr_+^\Gamma:=\Lscr_+\cap\range\Gamma$. This is due to a simple geometric result. More precisely, the adjoint operator of $\Gamma$ in (\ref{map_Gamma}) is given by (cf. \cite{FPZ-10})
\begin{equation}\label{Gamma_star}
\begin{split}
\Gamma^*:\Hfrak_n & \to C(\Tbb;\Hfrak_m) \\
X & \mapsto G^*XG,
\end{split}
\end{equation}
and we have the relation
\begin{equation}\label{Range_Gamma_ortho}
\begin{split}
\left(\range\Gamma\right)^\perp & =\ker\Gamma^* \\
 & =\left\{X\in\Hfrak_n\,:\,G^*(z)XG(z)=0,\ \forall z\in\Tbb\right\}.
\end{split}
\end{equation}
Hence for any $\Lambda\in\Lscr_+$, we have the orthogonal decomposition
\[\Lambda=\Lambda^\Gamma+\Lambda^\perp\]
with $\Lambda^\Gamma\in\range\Gamma$ and $\Lambda^\perp$ in the orthogonal complement. In view of (\ref{Range_Gamma_ortho}), the part $\Lambda^\perp$ does not contribute to the function value of $G^*\Lambda G$ on the unit circle, and we simply have
\[\Lscr_+^\Gamma=\Pi_{\range\Gamma}\Lscr_+,\]
where $\Pi_{\range\Gamma}$ denotes the orthogonal projection operator onto the linear space $\range\Gamma$.

Define next a family of spectral densities
\begin{equation}\label{family_spectra}
\Sscr:=\left\lbrace\,\Phi(\psi,\Lambda)=\psi(G^*\Lambda G)^{-1}\,:\,\psi\in C_+(\Tbb),\ \Lambda\in\Lscr_+^\Gamma\,\right\rbrace.
\end{equation}
Our problem is formulated as follows.

\begin{problem}\label{spec_estimation}
	Given the filter bank $G(z)$ in (\ref{trans_func}) and the matrix $\Sigma\in\range_+\Gamma:=\range\Gamma\,\cap\Hfrak_{+,n}$, find all the spectral densities $\Phi$ in the family $\Sscr$ such that (\ref{mmt_constraint}) holds.
\end{problem}

The motivation for choosing such a family $\Sscr$ lies in the observation that for a fixed $\psi$, the solution of the optimization problem
\begin{equation}
\underset{\Phi\in\Sfrak_m}{\text{maximize}}\ \int\psi\log\det\Phi\quad \text{subject to } (\ref{mmt_constraint})
\end{equation}
has exactly that form, where the matrix $\Lambda$ appears in the dual problem
\begin{equation}\label{J_dual}
\underset{\Lambda\in\Lscr_+^\Gamma}{\text{minimize}}\quad\Jbb_\psi(\Lambda)=\langle\Lambda,\Sigma\rangle-\int\psi\log\det(G^*\Lambda G).
\end{equation}
The scalar function $\psi$ encodes some \emph{a priori} information\footnote{Since $\Phi$ is a matrix-valued spectral density, the prior should be understood as $\psi I_m$.} that we have on the solution density $\Phi$. This optimization problem has been well studied in \cite{avventi2011spectral}, which can be seen as a multivariate generalization of the scalar problem investigated in \cite{Georgiou-L-03}. An important point in the optimization approach is that the optimal dual variable $\Lambda$ does not lie on the boundary of the feasible set. Hence it is a stationarity point of the function $\Jbb_\psi$ and the stationarity condition guarantees the moment constraint \eqref{mmt_constraint}. We shall next approach Problem \ref{spec_estimation} in a different way. Essentially, we want to treat the stationarity equation $\nabla \Jbb_\psi(\Lambda)=0$ directly, and to solve it via successive approximation of the prior $\psi$. We will show that the desired solution parameter can be achieved by solving an ordinary differential equation given the initial condition. To this end, we need first to establish a further result on the well-posedness of Problem \ref{spec_estimation}.

\section{Further result of well-posedness}\label{sec:well-posed}

Consider the map
\begin{equation}\label{mmt_map}
\begin{split}
f:\,D:=C_+(\Tbb)\times \Lscr_+^\Gamma & \to \range_+\Gamma \\
(\psi,\Lambda) & \mapsto\int G\psi(G^*\Lambda G)^{-1}G^*.
\end{split}
\end{equation}
Given $\Sigma\in\range_+\Gamma$, we aim to solve the equation
\begin{equation}\label{func_eqn}
f(\psi,\Lambda)=\Sigma,
\end{equation}
which is in fact equivalent to the stationarity condition $\nabla\Jbb_\psi(\Lambda)=0$ of the function in (\ref{J_dual}) when $\psi$ is fixed.

The map $f$ has a nice property. As shown in \cite{zhu2017parametric}, for a fixed $\psi\in C_+(\Tbb)$, the section of the map
\begin{equation}\label{omega_map}
\omega(\,\cdot\,):=f(\psi,\,\cdot\,):\,\Lscr_+^\Gamma\to\range_+\Gamma
\end{equation}
is a diffeomorphism\footnote{The word ``diffeomorphism'' in this paper should always be understood in the $C^1$ sense.}. This means that the map above is (at least) of class $C^1$, and its Jacobian, which contains all the partial derivatives of $f$ w.r.t. its second argument, vanishes nowhere in the set $\Lscr_+^\Gamma$. This implies that the solution map
\begin{equation*}
s:\,(\psi,\Sigma)\mapsto\Lambda
\end{equation*}
is well defined, that is, for any fixed $\psi$, there exists a unique $\Lambda$ such that \eqref{func_eqn} holds. Moreover, the map $s(\psi,\,\cdot\,):\,\range_+\Gamma\to\Lscr_+^\Gamma$ is continuous. We shall next show the well-posedness in the other respect, namely continuity of the map
\begin{equation}\label{imp_func_of_psi}
s(\,\cdot\,,\Sigma):\,C_+(\Tbb)\to\Lscr_+^\Gamma
\end{equation} 
when $\Sigma$ is held fixed. Note that continuity here is to be understood in the metric space setting. Clearly, it is equivalent to consider solving the functional equation (\ref{func_eqn}) for $\Lambda$ in terms of $\psi$ when its right-hand side is fixed, which naturally falls into the scope of the implicit function theorem.

\subsection{Proof of continuity with respect to the prior function}\label{subsec:cont_prior}

We first show that our moment map $f$ in \eqref{mmt_map} is of class $C^1$ on its domain $D$. According to \cite[Proposition 3.5, p.~10]{lang1999fundamentals}, it is equivalent to show that the two partial derivatives of $f$ exist and are continuous in $D$. More precisely, the partials evaluated at a point are understood as linear operators between two underlying vector spaces
\begin{equation}
\begin{split}
f'_1: & \,D\to L(C(\Tbb),\range\Gamma), \\
f'_2: & \,D\to L(\range\Gamma,\range\Gamma).
\end{split}
\end{equation}
The symbol $L(X,Y)$ denotes the vector space of continuous linear operators between two Banach spaces $X$ and $Y$, which is itself a Banach space. It is then easy to verify the following.
\begin{equation}\label{f_1partial}
\begin{split}
f'_1(\psi,\Lambda):\,C(\Tbb) & \to\range\Gamma \\
\delta\psi & \mapsto\int G\delta\psi(G^*\Lambda G)^{-1}G^*
\end{split}
\end{equation}
Clearly, the above operator does not depend on $\psi$ due to linearity. We also have $f'_2(\psi,\Lambda):\,\range\Gamma \to\range\Gamma$
\begin{equation}\label{f_2partial}
\delta\Lambda \mapsto -\int G\psi(G^*\Lambda G)^{-1}(G^*\delta\Lambda G)(G^*\Lambda G)^{-1}G^*.
\end{equation}

We need some lemmas. Notice that convergence of a sequence of continuous functions on a fixed interval $[a,b]\subset\Rbb$ will always be understood in the max-norm
\begin{equation}\label{max-norm}
\|f\|:=\max_{t\in[a,b]}|f(t)|.
\end{equation}
For $m\times n$ matrix valued continuous functions in one variable, define the norm as
\begin{equation}\label{max-Frob-norm}
\|M\|:=\max_{t\in[a,b]}\|M(t)\|_F
\end{equation}
It is easy to verify that convergence in the norm (\ref{max-Frob-norm}) is equivalent to element-wise convergence in the max-norm (\ref{max-norm}). 

\begin{lemma}
For a $n\times p$ matrix continuous function $M(\theta)$ on $[-\pi,\pi]$, the inequality holds for the Frobenius norm
\begin{equation}\label{inequ_int_norm}
\left\|\int M(\theta)\right\|_F\leq\sqrt{np}\int\|M(\theta)\|_F.
\end{equation}
\end{lemma}

\begin{proof}
Let $m_{jk}(\theta)$ be the $(j,k)$ element of $M(\theta)$. Then we have
\begin{equation}
\begin{split}
\left\|\int M(\theta)\right\|_F^2 & =\sum_{j,k} \left|\int m_{jk}(\theta)\right|^2 \leq np \max_{j,k} \left|\int m_{jk}(\theta)\right|^2 \\
 & \leq np \max_{j,k} \left(\int |m_{jk}(\theta)|\right)^2 \\
 & \leq np \left(\int\|M(\theta)\|_F\right)^2
\end{split}
\end{equation}
where the third inequality holds because $|m_{jk}(\theta)|\leq \|M(\theta)\|_F$ for any $j,k$.
\end{proof}

\begin{lemma}\label{lem_uniform_converg}
If a sequence $\{\Lambda_k\}\subset\Lscr_+^\Gamma$ converges to $\Lambda\in\Lscr_+^\Gamma$, then the sequence of functions $\{(G^*\Lambda_kG)^{-1}\}$ converges to $(G^*\Lambda G)^{-1}$ in the norm $(\ref{max-Frob-norm})$.
\end{lemma}
\begin{proof}
From \cite[Lemma 10]{zhu2017parametric}, there exists $\mu>0$ such that for any $k$ and $\theta\in[-\pi,\pi]$, $G^*\Lambda_kG\geq\mu I$. Hence we have
\begin{equation}
\begin{split}
& \|(G^*\Lambda_kG)^{-1}-(G^*\Lambda G)^{-1}\|_F \\
= & \|(G^*\Lambda_kG)^{-1}G^*(\Lambda-\Lambda_k)G(G^*\Lambda G)^{-1}\|_F \\
 \leq & \,\kappa^2\mu^{-2}G_{\max}^2\|\Lambda_k-\Lambda\|_F\to0,
\end{split}
\end{equation}
where $\kappa$ here and in the sequel is a constant of norm equivalence $\|\cdot\|_F\leq\kappa\|\cdot\|_2$, the constant $G_{\max}:=\max_{\theta\in[-\pi,\pi]}\|G(e^{i\theta})\|_F$, and we have used submultiplicativity of the Frobenius norm.
\end{proof}

\begin{proposition}\label{prop_f_C1}
The map $f$ in $(\ref{mmt_map})$ is of class $C^1$.
\end{proposition}
\begin{proof}
We show that the two partial derivatives of $f$ are continuous in its domain.
Consider the partial w.r.t. the first argument \eqref{f_1partial}. Let the sequence $\{(\psi_k,\Lambda_k)\}\subset D$ converge in the product topology to $(\psi,\Lambda)\in D$, that is, $\psi_k\to\psi$ in the max-norm and $\Lambda_k\to\Lambda$ in any matrix norm. We need to show that
\[f'_1(\psi_k,\Lambda_k)\to f'_1(\psi,\Lambda).\]
in the operator norm. Indeed, we have
\begin{equation}
\begin{split}
 & \|f'_1(\psi_k,\Lambda_k)-f'_1(\psi,\Lambda)\| \\ = & \sup_{\|\delta\psi\|=1}\left\|\int G\delta\psi\left[\,(G^*\Lambda_kG)^{-1}-(G^*\Lambda G)^{-1}\,\right]G^*\right\|_F \\
\leq & \, nG_{\max}^2\left\|(G^*\Lambda_kG)^{-1}-(G^*\Lambda G)^{-1}\right\|\to0.
\end{split}
\end{equation}
where we have used the inequality (\ref{inequ_int_norm}) and Lemma \ref{lem_uniform_converg}.

For the partial derivative of $f$ w.r.t. the second argument \eqref{f_2partial}, let us set $\delta\Phi(\psi,\Lambda;\delta\Lambda)=\Phi(\psi,\Lambda)(G^*\delta\Lambda G)(G^*\Lambda G)^{-1}$ to ease the notation.
Through similar computation, we arrive at
\begin{equation}
\begin{split}
 & \|f'_2(\psi_k,\Lambda_k)-f'_2(\psi,\Lambda)\| \\ 
\leq & \,\sup_{\|\delta\Lambda\|=1} nG_{\max}^2\left\|\delta\Phi(\psi_k,\Lambda_k;\delta\Lambda)-\delta\Phi(\psi,\Lambda;\delta\Lambda)\right\|\to0.
\end{split}
\end{equation}
The limit tends to $0$ because the part
\begin{equation}
\begin{split}
 & \sup_{\|\delta\Lambda\|=1} \left\|\delta\Phi(\psi_k,\Lambda_k;\delta\Lambda)-\delta\Phi(\psi,\Lambda;\delta\Lambda)\right\| \\
= & \max_{\substack{\|\delta\Lambda\|=1, \\ \theta\in[-\pi,\pi]}} \left\|\delta\Phi(\psi_k,\Lambda_k;\delta\Lambda)-\delta\Phi(\psi,\Lambda;\delta\Lambda)\right\|_F \\
= & \max_{\substack{\|\delta\Lambda\|=1, \\ \theta\in[-\pi,\pi]}} \left\|\delta\Phi(\psi_k,\Lambda_k;\delta\Lambda)-\Phi(\psi_k,\Lambda_k)(G^*\delta\Lambda G)(G^*\Lambda G)^{-1}\right. \\
& \qquad \left.+\Phi(\psi_k,\Lambda_k)(G^*\delta\Lambda G)(G^*\Lambda G)^{-1}-\delta\Phi(\psi,\Lambda;\delta\Lambda)\right\|_F \\
\leq & \max_{\substack{\|\delta\Lambda\|=1, \\ \theta\in[-\pi,\pi]}} \left(\|\Phi(\psi_k,\Lambda_k)\|_F\|(G^*\Lambda_k G)^{-1}-(G^*\Lambda G)^{-1}\|_F\right. \\
& \quad \left.+\|\Phi(\psi_k,\Lambda_k)-\Phi(\psi,
 \Lambda)\|_F\|(G^*\Lambda G)^{-1}\|_F\right)\|G^*\delta\Lambda G\|_F \\
\leq & \,\kappa\mu^{-1}G_{\max}^2\left(K_\psi\|(G^*\Lambda_k G)^{-1}-(G^*\Lambda G)^{-1}\|\right. \\
& \qquad \qquad \qquad \qquad \qquad \qquad \left.+\|\Phi(\psi_k,\Lambda_k)-\Phi(\psi,\Lambda)\|\right).
\end{split}
\end{equation}
Note that $\|\psi_k\|\leq K_{\psi}$ for some $K_\psi>0$ uniformly in $k$ because $\psi_k\to\psi$. Also, $\Phi(\psi_k,\Lambda_k)\to\Phi(\psi,\Lambda)$ is a simple consequence of Lemma \ref{lem_uniform_converg} and the fact
\[f_kg_k\to fg\text{ if }f_k\to f,\ g_k\to g.\]

Finally, the claim of the proposition follows from \cite[Proposition 3.5, p.~10]{lang1999fundamentals}.
\end{proof}

We are now in a place to state the main result of this section.

\begin{theorem}\label{thm_cont_s}
For a fixed $\Sigma\in\range_+\Gamma$, the implicit function $s(\,\cdot\,,\Sigma)$ in $(\ref{imp_func_of_psi})$ is of class $C^1$.
\end{theorem}

\begin{proof}
The assertion follows directly from the Banach space version of the implicit function theorem (see, e.g., \cite[Theorem 5.9, p.~19]{lang1999fundamentals}), because restrictions of $s(\,\cdot\,,\Sigma)$ must coincide with those locally defined, continuously differentiable implicit functions, which exist around every $\psi\in C_+(\Tbb)$ following from Proposition \ref{prop_f_C1} and the fact that the partial $f'_2(\psi,\Lambda)$ is a vector space isomorphism everywhere in $D$.
\end{proof}

\subsection{Reformulation in terms of the spectral factor}\label{subsec:reform}

Problem \ref{spec_estimation} can be reformulated in terms of the spectral factor of $G^*\Lambda G$ for $\Lambda\in\Lscr_+^\Gamma$. Though it may appear slightly more complicated, this reformulation is preferred from a numerical viewpoint, as the Jacobian of the new map corresponding to (\ref{mmt_map}) will have a smaller condition number when the solution is close to the boundary of the feasible set. This point has been illustrated in \cite{avventi2011spectral,enqvist2001homotopy} (see also later in Subsection \ref{subsec:comp-inv-Jac}). We shall first introduce a diffeomorphic spectral factorization.

According to \cite[Lemma 11.4.1]{FPZ-10}, given $\Lambda\in\Lscr_+^\Gamma$, the continuous spectral density $G^*\Lambda G$ admits a unique right outer spectral factor, i.e.,
\begin{equation}\label{W_Lambda}
G^*\Lambda G=W_\Lambda^*W_\Lambda.
\end{equation}
Furthermore, such a factor can be expressed in terms of the matrix $P\in\Hfrak_n$, the unique stabilizing solution of the Discrete-time Algebraic Riccati Equation (DARE)
\begin{equation}\label{DARE}
X = A^{*}X A- A^{*}X B(B^{*}X B)^{-1}B^{*}X A+\Lambda,
\end{equation}
as
\begin{equation}\label{W_spec_factor}
W_{\Lambda} (z) = L^{-*}B^{*}PA(zI -A)^{-1}B + L,
\end{equation}
where $L$ is the right (lower-triangular) Cholesky factor of the positive matrix $B^{*}PB(=L^*L)$.

Next, following the lines of \cite{avventi2011spectral}, let us introduce a change of variables by letting
\begin{equation}\label{C_factor}
C:=L^{-*}B^{*}P.
\end{equation}
Then it is not difficult to recover the relation $L=CB$. In this way, the spectral factor (\ref{W_spec_factor}) can be rewritten as
\begin{equation}\label{W_factor_C}
\begin{split}
W_{\Lambda} (z) & = CA(zI -A)^{-1}B + CB \\
 & =zCG,
\end{split}
\end{equation}
where the second equality holds because of the identity $A(zI-A)^{-1}+I=z(zI-A)^{-1}$. In view of this, the factorization (\ref{W_Lambda}) can then be rewritten as
\begin{equation}\label{spec_fact_Lambda}
G^*\Lambda G=G^*C^*CG, \quad \forall z\in\Tbb.
\end{equation}

As explained in \cite[Section A.5.5]{avventi2011spectral}, it is possible to build a bijective change of variables from $\Lambda$ to $C$ by carefully choosing the set where the ``factor'' $C$ lives. More precisely, let the set $\Cscr_+\subset\Cbb^{m\times n}$ contain those matrices $C$ that satisfy the following two conditions:
\begin{itemize}
	\item $CB$ is lower triangular with real and positive diagonal entries;
	\item $A-B(CB)^{-1}CA$ has eigenvalues in the open unit disk.
\end{itemize}
Define 
the map
\begin{equation}\label{h_map_spec_fact}
\begin{split}
h:\,\Lscr_+^\Gamma & \to\Cscr_+ \\
\Lambda & \mapsto C\textrm{ via } (\ref{C_factor}).
\end{split}
\end{equation}
It has been shown in \cite{avventi2011spectral} that the map $h$ is a \emph{homeomorphism} with an inverse
\begin{equation}\label{h_inverse}
\begin{split}
h^{-1}:\,\Cscr_+ & \to\Lscr_+^\Gamma \\
C & \mapsto \Lambda:=\Pi_{\range\Gamma}(C^*C),
\end{split}
\end{equation}
where $\Pi_{\range\Gamma}$ denotes the orthogonal projection operator onto $\range\Gamma$. This result has been further strengthened in \cite{zhu2017parametric}, as we quote below.

\begin{theorem}[\hspace{1sp}\cite{zhu2017parametric}]\label{thm_diffeo_fact}
	The map $h$ of spectral factorization is a diffeomorphism.
\end{theorem}


Now we can introduce the moment map $g:\,C_+(\Tbb)\times\Cscr_+\to\range_+\Gamma$ parametrized in the new variable $C$ as
\begin{equation}\label{mmt_map_C}
g(\psi,C):=f(\psi,h^{-1}(C))=\int G\psi(G^*C^*CG)^{-1}G^*,
\end{equation}
and the sectioned map when $\psi\in C_+(\Tbb)$ is held fixed
\begin{equation}\label{tau_map}
\tau:=\omega\circ h^{-1}:\,\Cscr_+\to\range_+\Gamma,
\end{equation}
where $\omega$ has been defined in (\ref{omega_map}). A corresponding problem is formulated as follows.

\begin{problem}\label{spec_estima_C}
	Given the filter bank $G(z)$ in $(\ref{trans_func})$, the matrix $\Sigma\in\range_+\Gamma$, and an arbitrary $\psi\in C_+(\Tbb)$, find the parameter $C\in\Cscr_+$ such that
	\begin{equation}\label{tau_of_C=Sigma}
	\tau(C)=\Sigma.
	\end{equation}
\end{problem}

The next corollary is an immediate consequence of Theorems \ref{thm_cont_s} and \ref{thm_diffeo_fact} and is stated without proof.

\begin{corollary}\label{cor_wellposed_C}
The map $\tau$ in $(\ref{tau_map})$ is a diffeomorphism. Moreover, if we fix the matrix $\Sigma$ and allow the prior $\psi$ to vary, then the solution map 
\begin{equation}
h\circ s(\,\cdot\,,\Sigma):\,C_+(\Tbb)\to\Cscr_+
\end{equation}
is of class $C^1$.
\end{corollary}

Therefore, Problem \ref{spec_estima_C} is also well-posed exactly like Problem \ref{spec_estimation}. Next we shall elaborate how to solve the equation (\ref{tau_of_C=Sigma}) numerically using a continuation method. 

\begin{remark}
Notice that we can also reparametrize the dual optimization problem \eqref{J_dual} in the domain of spectral factors by considering the cost function $\Jbb_\psi\circ h^{-1}$ with a fixed $\psi$ as done in \cite{avventi2011spectral}. However, convexity is lost in this way since neither the cost function nor the feasible set $\Cscr_+$ is convex. In this case, convergence of gradient-based descent algorithms starting from an arbitrary feasible point seems hard to guarantee.\footnote{A local convergence result can be found in \cite[Chapter 4]{Zhu_PhDthesis}.} A similar continuation solver can be built in the style of \cite{enqvist2001homotopy} but we shall not insist on this point here.
\end{remark}

\section{A numerical continuation solver}
\label{sec:numeric}

In this section we shall work with coordinates in the sense explained next since it is convenient for analysis. Specifically, we know from \cite[Proposition 3.1]{FPZ-12} that $\range\Gamma\subset\Hfrak_n$ is a linear space with real dimension $M:=m(2n-m)$. At the same time, the set $\Cscr_+$ is an open subset of the linear space 
\begin{equation}\label{C_space}
\begin{aligned}
\Cfrak:=\left\{\,C\in\Cbb^{m\times n}\right.\,: & \, CB \textrm{ is lower triangular} \\
 & \left.\textrm{with real diagonal entries}\,\right\},
\end{aligned}
\end{equation}
whose real dimension coincides with that of $\range\Gamma$ (cf. the proof of \cite[Theorem A.5.5]{avventi2011spectral}). We can hence choose orthonormal bases $\{\Lambdab_1,\Lambdab_2,\dots,\Lambdab_M\}$ and $\{\Cb_1,\dots,\Cb_M\}$ for $\range\Gamma$ and $\Cfrak$, respectively, and parametrize $\Lambda\in\Lscr_+^\Gamma$ and $C\in\Cscr_+$ as
\begin{equation}\label{Lambda_C_coord}
\begin{split}
\Lambda(x) & =x_1\Lambdab_1+x_2\Lambdab_2+\cdots+x_M\Lambdab_M, \\
C(y) & =y_1\Cb_1+y_2\Cb_2+\cdots+y_M\Cb_M,
\end{split}
\end{equation}
for some $x_j,y_j\in\Rbb$, $j=1,\dots,M$. We shall then introduce some abuses of notation and make no distinction between the variable and its coordinates. For example, $f(\psi,x)$ is understood as $f(\psi,\Lambda(x))$ defined previously and similarly, $\tau(y)$ means $\tau(C(y))$.

Instead of dealing with one particular equation (\ref{tau_of_C=Sigma}), a continuation method (cf.~\cite{allgower1990continuation}) aims to solve a family of equations related via a homotopy, i.e., a continuous deformation. In our context, there are two ways to construct different homotopies. One is to deform the covariance data $\Sigma$ and study the equation (\ref{tau_of_C=Sigma}) for a fixed $\psi$. Such an argument has been used extensively in \cite{Georgiou-06,georgiou2005solution}. Here we shall adopt an alternative, that is, deforming the prior function $\psi$ while keeping the covariance matrix fixed, which can be seen as a multivariate generalization of the argument in \cite[Section 4]{enqvist2001homotopy}. An advantage to do so is that we can obtain a family of matrix spectral densities that are consistent with the covariance data.

The set $C_+(\Tbb)$ is easily seen to be convex. One can then connect $\psi$ with the constant function $\oneb$ (taking value $1$ on $\Tbb$) via the line segment
\begin{equation}\label{path_convex}
p(t)=(1-t)\oneb+t\psi,\quad t\in U=[0,1],
\end{equation}
and construct a convex homotopy $U\times \Cscr_+ \to \range_+\Gamma$ given by
\begin{equation}
(t,y)\mapsto g(p(t),y).
\end{equation}
Now let the covariance matrix $\Sigma\in\range_+\Gamma$ be fixed whose coordinate vector is $x_\Sigma$, and consider the family of equations
\begin{equation}\label{family_eqns}
g(p(t),y)=x_\Sigma
\end{equation}
parametrized by $t\in U$. By Corollary \ref{cor_wellposed_C}, we will have a continuously differentiable solution path in the set $\Cscr_+$
\begin{equation}\label{path_solution}
y(t)=h(s(p(t),x_\Sigma)).
\end{equation}
Moreover, differentiating (\ref{family_eqns}) on both sides w.r.t $t$, one gets
\[g'_1(p(t),y(t);p'(t))+g'_2(p(t),y(t);y'(t))=0,\]
where $p'(t)\equiv\psi-\oneb$ independent of $t$, and the partial derivatives are given by
\begin{subequations}
\begin{align}
g'_1(\psi,y) & = f'_1(\psi,h^{-1}(y)), \label{g'_1} \\
g'_2(\psi,y) & = f'_2(\psi,h^{-1}(y))J_{h^{-1}}(y). \label{g'_2}
\end{align}
\end{subequations}
The semicolon notation here has the meaning e.g., $g'_1(\,\cdot\,,\,\cdot\,;\xi):=g'_1(\,\cdot\,,\,\cdot\,)(\xi)$, i.e., the operator $g'_1(\,\cdot\,,\,\cdot\,)$ applied to the function $\xi$.
The symbol $J_{h^{-1}}(y)$ means the Jacobian matrix of $h^{-1}$ evaluated at $y$.
Hence the path $y(t)$ is a solution to the initial value problem (IVP)
\begin{equation}\label{IVP}
\left \{
  \begin{aligned}
    y'(t) & =-\left[\,g'_2(p(t),y(t))\,\right]^{-1}g'_1(p(t),y(t);p'(t)) \\
    y(0) & =y^{(0)}
  \end{aligned} \right..
\end{equation}
Notice that the partial $g'_2$ is a finite-dimensional Jacobian matrix which is invertible everywhere in $D$ since both terms on the right hand side of (\ref{g'_2}) are nonsingular (cf.~\cite{zhu2017parametric}). From classical results on the uniqueness of solution to an ODE, we know that the IVP formulation and (\ref{family_eqns}) are in fact equivalent.

The initial value $y^{(0)}$ corresponds to $\psi=\oneb$, and it is the spectral factor of the so-called maximum entropy solution, i.e., solution to the problem
\begin{equation}
\underset{\Phi\in\Sfrak_m}{\text{maximize}}\ \int\log\det\Phi\quad \text{subject to } (\ref{mmt_constraint}).
\end{equation}
As has been worked out in \cite{georgiou2002spectral}, the above optimization problem has a unique solution $\Phi=(G^*\Lambda G)^{-1}$ with
\[\Lambda=\Sigma^{-1}B(B^*\Sigma^{-1}B)^{-1}B^*\Sigma^{-1},\]
from which the corresponding spectral factor $C$ can be computed as
\begin{equation}\label{C_0}
C=L^{-*}B^*\Sigma^{-1},
\end{equation}
where $L$ is the right Cholesky factor of $B^*\Sigma^{-1}B$. According to \cite{georgiou2002spectral}, such $C$ is indeed in the set $\Cscr_+$, i.e., $CB$ lower triangular and the closed-loop matrix is stable.


At this stage, any numerical ODE solver can in principle be used to solve the IVP and obtain the desired solution $y(1)$ corresponding to a particular prior $\psi$. However, this IVP is special in the sense that for a fixed $t$, $y(t)$ is a solution to a finite-dimensional nonlinear system of equations, for which there are numerical methods (such as Newton's method) that exhibit rapid local convergence properties, while a general-purpose ODE solver does not take this into account. Out of such consideration, a method called ``predictor-corrector'' is recommended in \cite{allgower1990continuation} to solve the IVP, which is reviewed next.

Suppose that for some $t\in U$ we have got a solution $y(t)$ and we aim to solve (\ref{family_eqns}) at $t+\delta t$ where $\delta t$ is a chosen step length. The predictor step is just numerical integration of the differential equation in (\ref{IVP}) using e.g., the Euler method
\begin{equation}\label{predict_Euler}
z(t+\delta t):=y(t)+v(t)\delta t,
\end{equation}
where $v(t):=-\left[\,g'_2(p(t),y(t))\,\right]^{-1}g'_1(p(t),y(t);p'(t))$.
The corrector step is accomplished by the Newton's method to solve (\ref{family_eqns}) initialized at the predictor $z(t+\delta t)$. If the new solution $y(t+\delta t)$ can be attained in this way, one can repeat such a procedure until reaching $t=1$.
The algorithm is summarized in the table.
 
\begin{algorithm}[H] 
\caption{Predictor-Corrector}
\label{alg:predict-correct}
\begin{algorithmic}
\STATE{Let $k=0$, $t=0$, and $y^{(0)}$ initialized as in (\ref{C_0})}
\STATE{Choose a sufficiently small step length $\delta t$}
\WHILE{$t\leq1$}
\STATE{Predictor: $z^{(k+1)}=y^{(k)}+v(t)\delta t$ the Euler step (\ref{predict_Euler})}
\STATE{Corrector: solve (\ref{family_eqns}) at $t+\delta t$ for $y^{(k+1)}$ initiated at $z^{(k+1)}$ using Newton's method }
\STATE{Update $t:=\min\{1,t+\delta t\}$, $k:=k+1$}
\ENDWHILE
\RETURN The last $y^{(k)}$ corresponding to $t=1$
\end{algorithmic}
\end{algorithm}

\subsection{Convergence analysis}\label{subsec:converge}

We are now left to determine the step length $\delta t$ so that the corrector step can converge and the algorithm can return the target solution $y(1)$ in a finite number of steps. We show next that one can choose a uniformly constant step length $\delta t$ such that the predictor $z^{(k)}$ will be close enough to the solution $y^{(k)}$ for the Newton's method to converge locally.\footnote{Notice that convergence results in \cite{allgower1990continuation} under some general assumptions do not apply here directly.} We shall need the next famous Kantorovich theorem which can be found in \cite[p.~421]{ortega2000iterative}.

\begin{theorem}[Kantorovich]
	Assume that $f:\, D\subset\Rbb^n\to\Rbb^n$ is differentiable on a convex set $D_0\subset D$ and that
	\[\|f'(x)-f'(y)\|\leq\gamma\|x-y\|,\quad\forall\,x,y\in D_0\]
	for some $\gamma>0$. Suppose that there exists an $x^{(0)}\in D_0$ such that $\alpha=\beta
	\gamma\eta\leq1/2$ for some $\beta,\eta> 0$ meeting
    \[\beta\geq\|f'(x^{(0)})^{-1}\|,\quad\eta\geq\|f'(x^{(0)})^{-1}f(x^{(0)})\|.\]
	Set
	\begin{subequations}
	\begin{align}
	    t^*=(\beta\gamma)^{-1}\left[\,1-(1-2\alpha)^{1/2}\,\right], \label{t*}\\
	    t^{**}=(\beta\gamma)^{-1}\left[\,1+(1-2\alpha)^{1/2}\,\right] \label{t**},
	    \end{align}
	\end{subequations}
	and assume that the closed ball $\overline{B}(x^{(0)},t^*)$ is contained in $D_0$.
	Then the Newton iterates
	\[x^{(k+1)} = x^{(k)}-f'(x^{(k)})^{-1}f(x^{(k)}),\quad k = 0, 1,\dots\]
	are well-defined, remain in $\overline{B}(x^{(0)},t^*)$, and converge to a solution $x$ of $f(x) = 0$ which is unique in $\overline{B}(x^{(0)},t^{**})\cap D_0$.
\end{theorem}

In order to apply the above theorem, we need to take care of the locally Lipschitz property. To this end, we shall first introduce a compact set in which we can take extrema of various norms.

\begin{lemma}\label{lem_compact_K}
There exists a compact set $K\subset\Cscr_+$ that contains the solution path $\{y(t)\,:\,t\in U\}$ indicated in $(\ref{path_solution})$ in its interior.

\end{lemma}
\begin{proof}
    We know from previous reasoning that the solution path is contained in the open set $\Cscr_+$ which is a subset of the \emph{finite}-dimensional vector space $\Cfrak$ in \eqref{C_space}. By continuity, the set $\{y(t)\}$ is easily seen to be compact, i.e., closed and bounded, and thus admits a compact neighborhood $K\subset\Cscr_+$. Such a neighborhood $K$ can be constructed explicitly as follows. Let $B(y(t))\subset\Cscr_+$ be an open ball centered at $y(t)$ such that its closure is also contained in $\Cscr_+$. Then the set	$\bigcup_{t\in U}\,B(y(t))$
	is an open cover of $\{y(t)\}$, which by compactness, has a finite subcover
	\[\bigcup_{k=1}^n\,B(y(t_k))\]
	whose closure can be taken as $K$.
\end{proof}

\begin{lemma}\label{lem_Lip_cont}
For a fixed $t\in U$, the derivative $g'_2(p(t),y)\in L(\Cfrak,\range\Gamma)$ is locally Lipschitz continuous in $y$ in any convex subset of the compact set $K$ constructed in Lemma~$\ref{lem_compact_K}$, where $p(t)$ is the line segment given in $(\ref{path_convex})$. Moreover, the Lipschitz constant can be made independent of $t$.
\end{lemma}
\begin{proof}
It is a well known fact that a continuously differentiable function is locally Lipschitz. Hence we need to check the continuity of the second-order derivative following from (\ref{g'_2})
\begin{equation}\label{g''_22_differential}
\begin{split}
g''_{22}(\psi,y;\delta y_1,\delta y_2)= & f''_{22}(\psi,h^{-1}(y);J_{h^{-1}}(y)\delta y_2,J_{h^{-1}}(y)\delta y_1) \\
 & +f'_2(\psi,h^{-1}(y))\frac{d}{dy}J_{h^{-1}}(y)(\delta y_2,\delta y_1).
\end{split}
\end{equation}
Here $\frac{d}{dy}J_{h^{-1}}(y)$ is the second order derivative of $h^{-1}$ evaluated at $y$ which is viewed as a bilinear function $\Cfrak\times\Cfrak\to\range\Gamma$. It is continuous in $y$ since the function (\ref{h_inverse}) is actually smooth (of class $C^{\infty}$).
Although this involves more tedious computations, one can also show that the second-order partial $f''_{22}(\psi,x)$ is continuous following the lines in Section \ref{sec:well-posed}. Therefore, for fixed $\delta y_1,\delta y_2$, the differential (\ref{g''_22_differential}) is continuous in $(\psi,y)$. Consider any convex subset $D_0\subset K$ with $y_1,y_2\in D_0$. By the mean value theorem we have
\begin{equation}
\begin{split}
 & \|g'_2(\psi,y_2)-g'_2(\psi,y_1)\| \\
 = & \left\|\int_{0}^{1}g''_{22}(\psi,y_1+\xi(y_2-y_1))d\xi(y_2-y_1)\right\| \\
 \leq & \max_{y\in K}\|g''_{22}(\psi,y)\|\|y_2-y_1\|.
\end{split}
\end{equation}
Now let us replace $\psi$ with $p(t)$. The local Lipschitz constant can be taken as
\begin{equation}
\gamma:=\max_{t\in U,\ y\in K}\|g''_{22}(p(t),y)\|.
\end{equation}

\end{proof}

Based on precedent lemmas, our main result in this section is stated as follows.

\begin{theorem}
	Algorithm $\ref{alg:predict-correct}$ returns a solution to $(\ref{family_eqns})$ for $t=1$ in a finite number of steps.
\end{theorem}
\begin{proof}
	At each step, our task is to solve (\ref{family_eqns}) for $t+\delta t$ from the initial point $z(t+\delta t)=y(t)+v(t)\delta t$ given in (\ref{predict_Euler}). The idea is to work in the compact set $K$ introduced in Lemma \ref{lem_compact_K}. The boundary of $K$ is denoted by $\partial K$ which is also compact.
	
	First, we show that the predictor $z(t+\delta t)$ will always stay in $K$ as long as the step length $\delta t$ is sufficiently small. Define
	\begin{align}
		c_1 & :=\min_{t\in U}\,d(y(t),\partial K), \\
		c_2 & :=\max_{t\in U}\|v(t)\|,
	\end{align}
	where $d(x,A):=\min_{y\in A}d(x,y)$ is the distance function from a point $x$ to a set $A$.
	Note that $c_1>0$ because all the points $\{y(t)\}$ are in the interior of $K$.
	Then we see that the condition
	\begin{equation}\label{delta_t1}
	\delta t<\frac{c_1}{c_2}:=\delta t_1
	\end{equation}
	is sufficient since in this way
	\[\|v(t)\delta t\|\leq c_2\delta t<c_1\leq d(y(t),\partial K),\ \forall t\in U,\]
	which implies that $z(t+\delta t)\in K$.
	The reason is that one can always go from $y(t)$ in the direction of $v(t)$ until the boundary of $K$ is hit.

	Secondly, we want to apply the Kantorovich Theorem to ensure convergence of the corrector step, i.e., the Newton iterates. The function $\psi=p(t+\delta t)$ is held fixed in the corrector step. The uniform Lipschitz constant $\gamma$ has been given in Lemma \ref{lem_Lip_cont}, and there are two remaining points:
	\begin{enumerate}[(i)]
	\item We need to take care of the constraint $\alpha=\beta\gamma\eta\leq1/2$. Clearly, we can simply take $\beta=\|g'_{2}(\psi,y^{(0)}_\mathrm{in})^{-1}\|$ and
	\[\eta=\|g'_{2}(\psi,y^{(0)}_\mathrm{in})^{-1}(g(\psi,y^{(0)}_\mathrm{in})-x_\Sigma)\|,\]
	where $y^{(0)}_\mathrm{in}=z(t+\delta t)$ is the initialized inner-loop variable.
		Define 
		\begin{align}
		c_3 & :=\max_{y\in K,t\in U}\|g'_2(p(t),y)^{-1}\|, \\
		c_4 & :=\max_{y\in K,t\in U}\|g''_{22}(p(t),y)\|,
		\end{align}
		and obviously we have $\beta\leq c_3$, $\eta\leq c_3\|g(\psi,y^{(0)}_\mathrm{in})-x_\Sigma\|$. Hence a sufficient condition is
		\[\|g(\psi,y^{(0)}_\mathrm{in})-x_\Sigma\|\leq\frac{1}{2c_3^2\gamma},\]
		and we need an estimate of the left hand side. The Taylor expansion of $g$ in its second argument is
		\begin{equation}\label{g_Taylor}
		\begin{split}
		g(\psi,y(t)+v(t)\delta t)= & g(\psi,y(t))+\delta tg'_2(\psi,y(t))v(t) \\
		& +\frac{\delta t^2}{2}\Bcal[v(t),v(t)],
		\end{split}
		\end{equation}
		where $\Bcal$ is the bilinear function determined by the second order partials. Due to linearity and the identity $\psi=p(t+\delta t)=p(t)+\delta tp'(t)$, the first term
		\begin{equation}
		\begin{split}
			g(p(t+\delta t),y(t)) & =g(p(t),y(t))+\delta tg'_1(p(t),y(t);p'(t)) \\
			 & =x_\Sigma+\delta tg'_1(p(t),y(t);p'(t)).
		\end{split}
		\end{equation}
		The matrix in the second term\footnote{Attention: $g'_2(p'(t),y(t))$ is an abuse of notation because $p'(t)=\psi-\oneb$ may not be in the domain of the functional. It should be understood as substituting $\psi$ with $p'(t)$ in the expression of $g'_2(\psi,y(t))$.}
		\begin{equation}
		g'_2(p(t+\delta t),y(t))=g'_2(p(t),y(t))+\delta tg'_2(p'(t),y(t))
		\end{equation}
		Substituting these two expressions into (\ref{g_Taylor}), we obtain a cancellation due to the definition of $v(t)$ after (\ref{predict_Euler}) and we have
		\begin{equation}
		\begin{split}
			g(\psi,y(t)+v(t)\delta t)-x_\Sigma= & \delta t^2g'_2(p'(t),y(t))v(t) \\
			& +\frac{\delta t^2}{2}\Bcal[v(t),v(t)] \\
	    \end{split}
		\end{equation}
		whose norm is less than $\delta t^2(c_5c_2+\frac{1}{2}c_2^2c_4)$,
		where
		\[c_5:=\max_{y\in K}\|g'_2(p'(t),y)\|.\]
		We end up having the sufficient condition
		\[\delta t^2(c_5c_2+\frac{1}{2}c_2^2c_4)\leq\frac{1}{2c_3^2\gamma}\implies\delta t\leq\delta t_2.\]
	
	\item We need to insure that the closed ball $\overline{B}(y^{(0)}_\mathrm{in},t^*)$ is also contained in $K$. Clearly, we only need to make $t^*\leq\min_{t\in U}\,d(y(t)+v(t)\delta t_1/2,\partial K)=:r_2$, where $\delta t_1$ is the uniform step determined in (\ref{delta_t1}). This can be done since by its definition (\ref{t*}), $t^*$
		tends to $0$ when the step length $\delta t\to0$. A sufficient condition is
		\[1-\sqrt{1-2\alpha}\leq c_6\gamma r_2\Longleftrightarrow\alpha\leq\frac{1}{2}(1-(1-c_6\gamma r_2)^2),\]
		provided that $1-c_6\gamma r_2>0$, where
		\[c_6:=\min_{y\in K,t\in U}\|g'_2(p(t),y)^{-1}\|.\]
		With the bound for $\beta$ and $\eta$ in the previous point, a more sufficient condition is
		\[\delta t^2\gamma c_3^2(c_5c_2+\frac{1}{2}c_2^2c_4)\leq\frac{1}{2}(1-(1-c_6\gamma r_2)^2),\]
		which implies $\delta t\leq\delta t_3$ (constant).
	
	\end{enumerate}
	
	At last we can just take $\delta t:=\min\{\,\delta t_1/2,\delta t_2,\delta t_3\,\}$. In this way, the Kantorovich theorem is applicable to ensure local convergence in each inner loop. The reasoning above is independent of $t$ and hence the step length is uniform. This concludes the proof.	
\end{proof}

\subsection{Computation of the inverse Jacobian}
\label{subsec:comp-inv-Jac}

The coordinate thinking is suitable for theoretical reasoning. However, when implementing the algorithm, it is better to work with matrices directly. In this section, we present a matricial linear solver adapted from \cite{RFP-09} (see also \cite{FMP-12}). Here we shall assume $\psi$ is rational and admits a factorization $\psi=\sigma\sigma^*$ where $\sigma$ is outer rational and hence realizable. A crucial step in the implementation of the numerical algorithm is the computation of the Newton direction $g'_2(\psi,y)^{-1}g(\psi,y)$, which amounts to solving the linear equation in $V$ given $C$ and $\psi$
\begin{equation}\label{Newton_eqn}
g'_2(\psi,C;V)=g(\psi,C)
\end{equation}
where $g(\psi,C)$ has been given in \eqref{mmt_map_C} and
\begin{subequations}
\begin{align}
g'_2(\psi,C;V) & = -\int G\psi(G^*C^*CG)^{-1}G^*(V^*C \nonumber\\
 & \qquad \qquad +C^*V)G(G^*C^*CG)^{-1}G^* \label{Jac_C_V_bb} \\
 & = -\int G\psi(CG)^{-1}\left[(G^*C^*)^{-1}G^*V^*\right. \nonumber\\
 & \qquad \qquad \left.+VG(CG)^{-1}\right](G^*C^*)^{-1}G^* \label{Jac_C_V_cc}
\end{align}
\end{subequations}
The cancellation of one factor $CG$ from (\ref{Jac_C_V_bb}) to (\ref{Jac_C_V_cc}) is precisely why the condition number of the Jacobian $g'_2$ is smaller than that of $f'_2$ in (\ref{f_2partial}) when $C$ tends to the boundary of $\Cscr_+$, i.e., when $CG(e^{i\theta})$ tends to be singular for some $\theta$. This point is illustrated in the next example.


\begin{example}[Reduction of the condition number of the Jacobian under the $C$ parametrization]\label{ex_CondNum}

First we need to find a matrix representation of the Jacobian $g'_2(\psi,C)$ which is a linear operator from $\Cfrak$ to $\range\Gamma$. Fix the orthonormal bases of the two vector spaces as in (\ref{Lambda_C_coord}). Then the $(j,k)$ element of the \emph{real} $M\times M$ Jacobian matrix corresponding to $g'_2(\psi,C)$ is just
\begin{equation}\label{Jac_mat_g}
\langle \Lambdab_j,g'_2(\psi,C;\Cb_k) \rangle.
\end{equation}
Similarly, the matrix representation of the Jacobian (\ref{f_2partial}) is
\begin{equation}\label{Jac_mat_f}
\langle \Lambdab_j,f'_2(\psi,\Lambda;\Lambdab_k) \rangle,\quad j,k=1,\dots,M.
\end{equation}
Next let us fix $\psi\in C_+(\Tbb)$ and $C\in\Cscr_+$, evaluate explicitly the Jacobian matrix \eqref{Jac_mat_g}, and compute its condition number. The same computation is done for the Jacobian matrix (\ref{Jac_mat_f}) evaluated at $(\psi,h^{-1}(C))$ where $h^{-1}$ has been defined in \eqref{h_inverse}. Comparison is made in such a way because taking $\Lambda=h^{-1}(C)$ will lead to the same spectrum in the moment map $f$ as that in $g$ due to the spectral factorization (\ref{spec_fact_Lambda}).

Our example is about the problem of matrix covariance extension mentioned in Section \ref{sec:problem} with $(A,B)$ matrices given in (\ref{A_B_cov_ext}) and the filter $G(z)$ in (\ref{trans_func_covext}). We set the dimension $m=2$, the maximal covariance lag $p=1$, and we have $n=m(p+1)=4$.

Let us treat the problem for real processes. Then $\range\Gamma$ is the $M=7$-dimensional vector space of symmetric block-Toeplitz matrices of the form
\[\begin{bmatrix}
\Lambda_0 & \Lambda_1^\top\\
\Lambda_1 & \Lambda_0
\end{bmatrix},\]
where $\Lambda_0,\Lambda_1$ are $2 \times 2$ blocks. An orthonormal basis of $\range\Gamma$ can be determined from the matrix pairs
\begin{equation}\label{basis_Range_Gamma}
\begin{split}
(\Lambda_0,\Lambda_1)\in\{\zerob\}\times
\left\{\begin{bmatrix}
1&0\\0&0
\end{bmatrix},
\begin{bmatrix}
0&1\\0&0
\end{bmatrix},
\begin{bmatrix}
0&0\\1&0
\end{bmatrix},
\begin{bmatrix}
0&0\\0&1
\end{bmatrix}\right\} \\
\bigcup \left\{\begin{bmatrix}
1&0\\0&0
\end{bmatrix},
\begin{bmatrix}
0&1\\1&0
\end{bmatrix},
\begin{bmatrix}
0&0\\0&1
\end{bmatrix}\right\}
\times \{\zerob\}
\end{split}
\end{equation}
after normalization. Here the bold symbol $\zerob$ denotes the $2\times2$ zero matrix.

On the other hand, the vector space $\Cfrak$ contains matrices of the shape
\[\begin{bmatrix}
C_1&C_0
\end{bmatrix},\]
where $C_1,C_0$ are also $2\times2$ blocks and $C_0$ is lower triangular. An orthonormal basis of $\Cfrak$ can be determined from the standard basis of $\Rbb^{m\times n}$ which is made up of matrices $E_{jk}$ whose elements are all zero except that on $(j,k)$ position it is one. A basis of $\Cfrak$ is obtained by excluding those $E_{jk}$ which constitute the (strict) upper triangular part of $C_0$. Notice that given $C\in\Cfrak$, $zCG$ is a matrix polynomial of degree $-p$.

The prior is chosen as a positive Laurent polynomial $\psi(z)=b(z)b(z^{-1})$ where the polynomial $b(z)=1-z^{-1}+0.89z^{-2}$ has roots $0.5\pm0.8i$ with a modulus $0.9434$. We choose the parameter
\begin{equation}\label{C_1}
C=\bmat
0.5&0.65&1&0\\
-2.2615&-1&2&1
\emat,
\end{equation}
which belongs to the set $\Cscr_+$, because the roots of $\det zCG$ are $0.9\pm0.4i$ with a modulus $0.9849$.

Integrals such as (\ref{Jac_C_V_cc}) are approximated with Riemann sums
\[\int F(\theta)\approx\frac{\Delta\theta}{2\pi}\sum_k F(\theta_k),\]
where $\{\theta_k\}$ are equidistant points on the interval $(-\pi,\pi]$ and the subinterval length $\Delta\theta=10^{-4}$. The resulting condition number of \eqref{Jac_mat_g} is $2.4674\times10^{5}$ while that of (\ref{Jac_mat_f}) is $3.8187\times10^8$.
\end{example}

In order to invert the Jacobian $g'_2$ at a given ``point'' $(\psi,C)$ without doing numerical integration, we first need to fix an orthonormal basis $\{\Cb_1,\dots,\Cb_M\}$ of $\Cfrak$ such that $\Cb_1=C/\|C\|$.\footnote{This is always possible by adding $C$ into any set of basis matrices and performing Gram-Schmidt orthonormalization starting from $C$.} Then one can obtain a basis $\{\Vb_1,\dots,\Vb_M\}$ of $\Cfrak$ such that for $k=1,\dots,M$
\[G^*(z)(\Vb_k^*C+C^*\Vb_k)G(z)>0,\quad\forall z\in\Tbb\]
by setting $\Vb_k=\Cb_k+r_kC$ for some $r_k\geq0$.
The procedure for solving (\ref{Newton_eqn}) is described as follows:
\begin{enumerate}[1)]
\item Compute $Y=g(\psi,C)$ and $\Yb_k=g'_2(\psi,C;\Vb_k)$.
\item Find $\alpha_k$ such that $Y=\sum\alpha_k\Yb_k$.
\item Set $V=\sum \alpha_k\Vb_k$.
\end{enumerate}

In order to obtain the coordinates $\alpha_k$ in Step 2, one needs to solve a linear system of equations whose coefficient matrix consists of inner products $\langle \Yb_k,\Yb_j\rangle$. The matrix is invertible because $\{\Yb_k\}$ are linearly independent, which is a consequence of the Jacobian $g'_2(\psi,C)$ being nonsingular.

The difficult part is Step 1 where we need to compute the integrals $g(\psi,C)$ and $g'_2(\psi,C;\Vb_k)$. Since we want to avoid numerical integration, we shall need some techniques from spectral factorization. Evaluation of the former integral was essentially done in the proof of \cite[Theorem 11.4.3]{FPZ-10}. More precisely, we have the expression
\[G(zCG)^{-1}=(zI-\varPi)^{-1}B(CB)^{-1},\]
where $\varPi:=A-B(CB)^{-1}CA$ is the closed-loop matrix which is stable.
With a state-space realization $(A_1,B_1,C_1)$ of the stable proper transfer function $\sigma G(zCG)^{-1}$, one then solves for $R$ a discrete-time Lyapunov equation
\[R-A_1RA_1^*=B_1B_1^*.\]
Finally the integral $g(\psi,C)=C_1RC_1^*$.

The integral $g'_2(\psi,C;\Vb_k)$ can be computed similarly. The only difference is that we need to compute a left outer factor $W(z)$ of
\[Z^*(z)+Z(z)>0\text{ on }\Tbb\]
where 
\begin{equation}
\begin{split}
Z(z) & =zVG(zCG)^{-1} \\
& =Vz(zI-\varPi)^{-1}B(CB)^{-1} \\
 & = V\varPi(zI-\varPi)^{-1}B(CB)^{-1}+VB(CB)^{-1}\\
\end{split}
\end{equation}
The factorization involves solving a DARE for the unique stabilizing solution, in terms of which the factor can be expressed. Such a procedure is standard (cf.~the appendix for details). Once we have the factor $W(z)$, a realization of the transfer function $\sigma G(zCG)^{-1}W$ can be obtained, and we can just proceed in the same way as computing $g(\psi,C)$.

\begin{example}
Let us continue Example \ref{ex_CondNum} with the prior $\psi$ and the parameter $C$ given. We begin the simulation by computing the covariance matrix $\Sigma=g(\psi,C)$ with the formula given in \eqref{mmt_map_C}. Then the maximum entropy solution can be obtained with \eqref{C_0}, which is our initial value of Algorithm \ref{alg:predict-correct}. The step length is set as $\delta t=0.1$ which is quite large but sufficient for convergence in this particular example. Our target parameter $C^{(1)}$ is certainly equal to the given $C$ in \eqref{C_1}. The simulation result is shown in Figures \ref{fig:result1} and \ref{fig:result2}, where the coordinates of the solution parameter are plotted against the variable $t\in[0,1]$. One can see that the solution curves are smooth.

\begin{figure}[!t]
\centering
\includegraphics[width=0.5\textwidth]{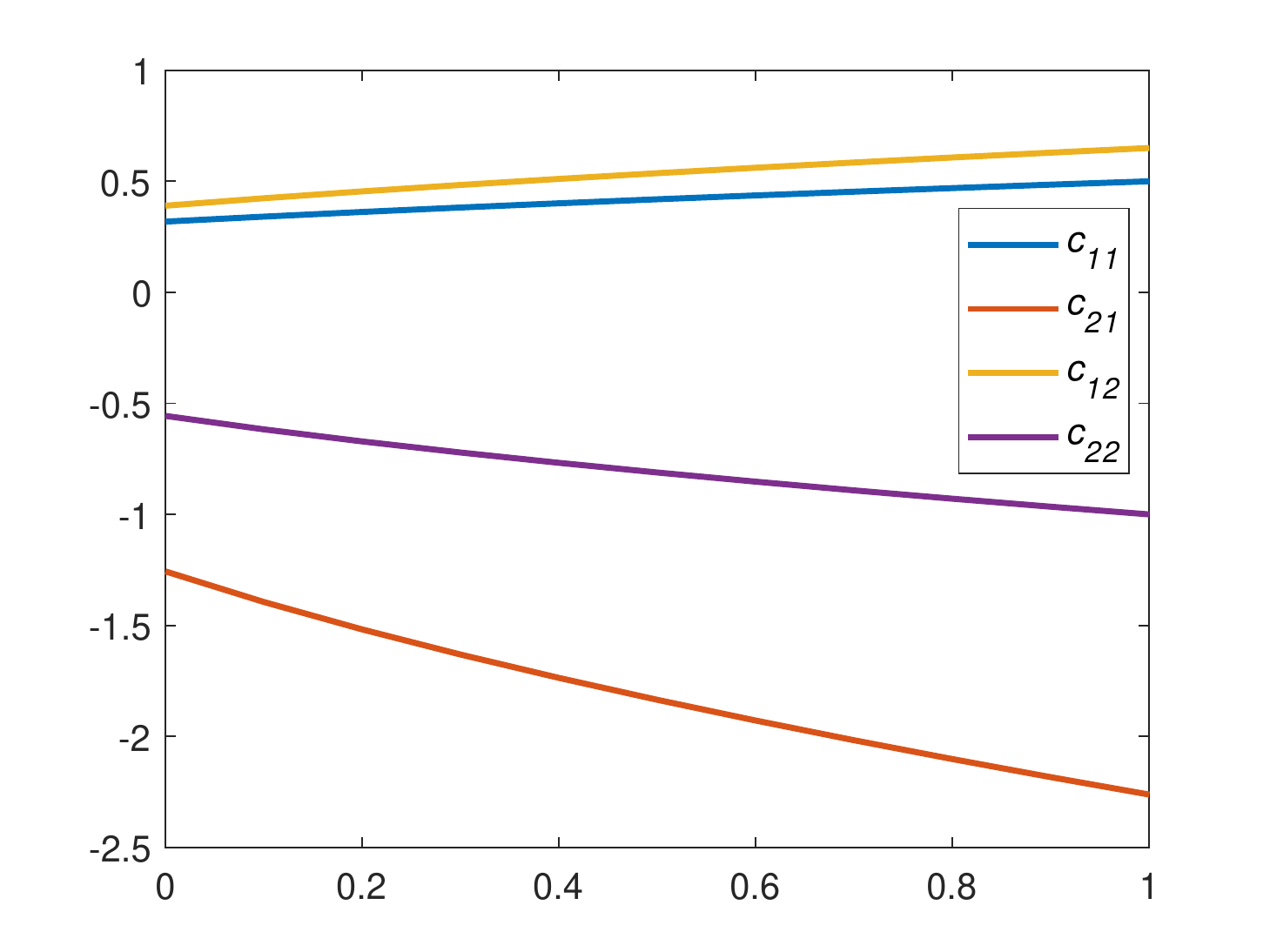}
\caption{Solution parameter (coordinates) against the variable $t$.}
\label{fig:result1}
\end{figure}

\begin{figure}[!t]
\centering
\includegraphics[width=0.5\textwidth]{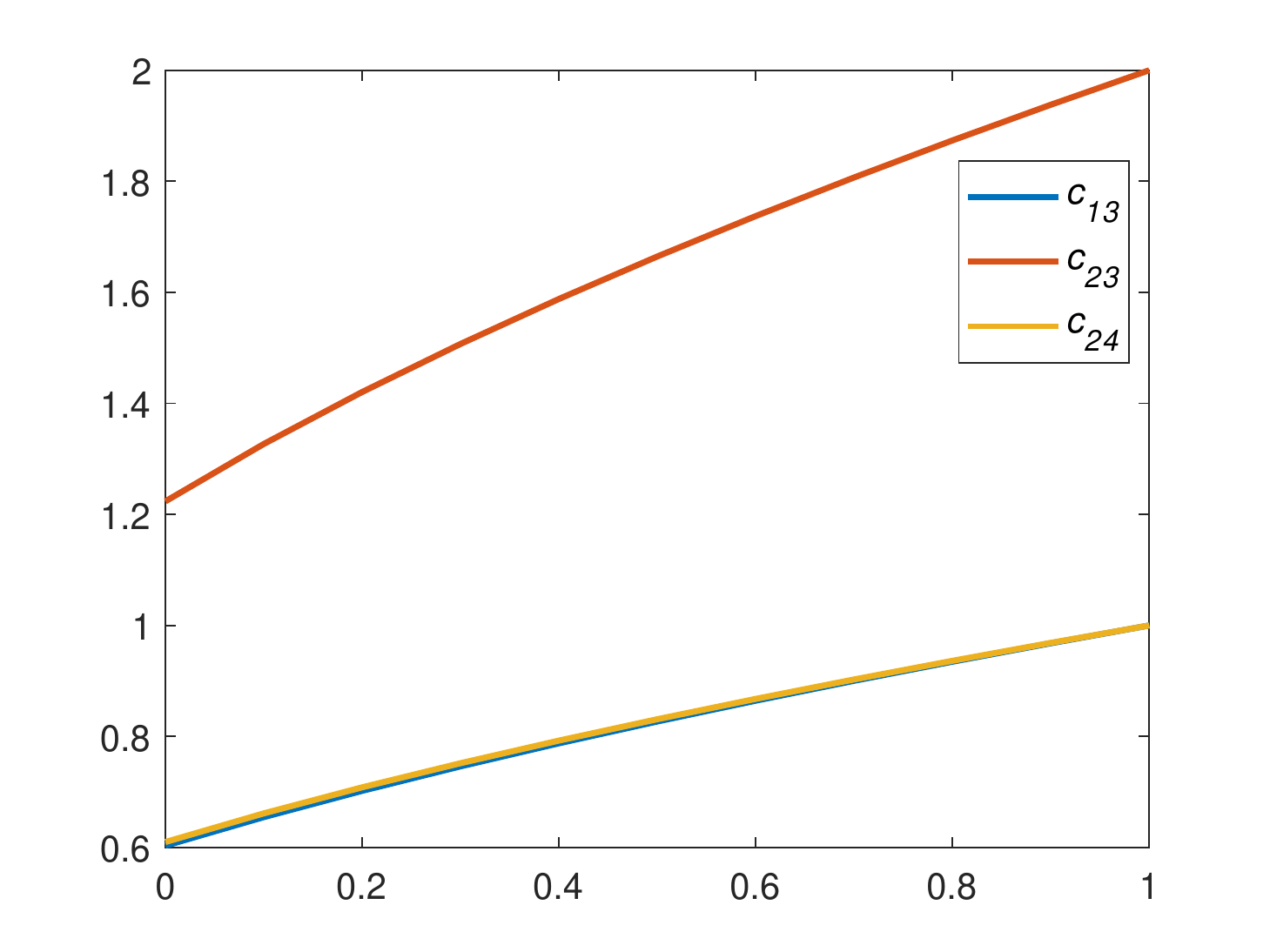}
\caption{Solution parameter (coordinates) against the variable $t$.}
\label{fig:result2}
\end{figure}

\end{example}

\section{Conclusions}
\label{sec:conclusions}

In this paper, we have addressed a spectral estimation problem formulated in a parametric fashion. The estimated state covariance matrix and the scalar prior function serve as data. When the prior is fixed, the problem is well-posed with respect to the covariance data according to \cite{zhu2017parametric}. Here we have shown continuous dependence of the solution parameter with respect to the prior function when the covariance matrix is held fixed. In this sense, we have completed the proof of well-posedness.

Moreover, we have indicated how to numerically solve the spectral estimation problem using a continuation method. Although the resulting algorithm seems more complicated than the optimization approach in \cite{avventi2011spectral}, it serves as a viable alternative with the benefit of obtaining a family of solutions parametrized by a real variable living on the unit interval that are consistent with the covariance data. While the problem is well-posed, in practice the Jacobian of the moment map may become ill-conditioned when the (solution) parameter goes near the boundary of the feasible set. Such a numerical issue can be alleviated if we carry out computations in the domain of spectral factors.

At last, we hope to generalize the results in this work to the open problem left in \cite{zhu2017parametric,Zhu-Baggio-17} when the prior function is matrix-valued. In that case, the candidate solution spectrum takes the form $(CG)^{-1}\Psi(CG)^{-*}$ with a matrix prior density $\Psi$ and some $C\in\Cscr_+$, in contrast to \eqref{family_spectra} where the prior $\psi$ is scalar-valued. One technical difficulty, namely uniqueness of the solution in that more general case remains to be tackled. One can expect that once well-posedness is established, the numerical continuation procedure to find the solution can be extended in a straightforward manner.

\appendix
\section*{From additive decomposition to spectral factorization}\label{sec:additive2fact}
Let $Z(z)=H(zI-F)^{-1}G+J$ with $F\in\Cbb^{n\times n}$ stable, $G\in\Cbb^{n\times m}$, $H\in\Cbb^{m\times n}$, and $J\in\Cbb^{m\times m}$. Suppose that $\Phi(z)=Z(z)+Z^*(z)>0$ for all $z\in\Tbb$. Set $R:=J+J^*>0$. Then one can write
\[\begin{split}
\Phi(z)=\bmat H(zI-F)^{-1} & I \emat\bmat0&G\\G^*&R\emat\bmat(z^{-1}I-F^*)^{-1}H^*\\I\emat,
\end{split}\]
which adding to the identity that holds for any Hermitian $P$
\[\begin{split}
0\equiv\bmat H(zI-F)^{-1}&I\emat\bmat FPF^*-P&FPH^*\\HPF^*&HPH^*\emat \\ \times\bmat(z^{-1}I-F^*)^{-1}H^*\\I\emat
\end{split}\]
yields
\[\begin{split}
\Phi(z)=\bmat H(zI-F)^{-1}&I\emat\bmat FPF^*-P&G+FPH^*\\G^*+HPF^*&R+HPH^*\emat \\ \times\bmat(z^{-1}I-F^*)^{-1}H^*\\I\emat.
\end{split}\]
Consequently, if $P$ is the unique stabilizing solution of the DARE
\[P=FPF^*-(G+FPH^*)(R+HPH^*)^{-1}(G^*+HPF^*)\]
such that $R+HPH^*>0$,
then one obtains the factorization
\begin{equation*}
\begin{split}
\bmat FPF^*-P&G+FPH^*\\G^*+HPF^*&R+HPH^*\emat=\bmat G+FPH^*\\R+HPH^*\emat \qquad\qquad\\
\times(R+HPH^*)^{-1}\bmat G^*+HPF^*&R+HPH^*\emat.
\end{split}
\end{equation*}
Taking $L$ as the Cholesky factor of $R+HPH^*(=LL^*)$, one gets a left outer factor of $\Phi(z)$ in this way
\begin{equation*}
\begin{split}
W(z) & =\begin{bmatrix}
H(zI-F)^{-1} & I
\end{bmatrix}
\begin{bmatrix}
G+FPH^*\\R+HPH^*
\end{bmatrix}
L^{-*} \\
 & =H(zI-F)^{-1}(G+FPH^*)L^{-*}+L.
\end{split}
\end{equation*}

\bibliographystyle{IEEEtran}
\bibliography{references}

\begin{IEEEbiography}[{\includegraphics[width=1in,height=1.25in,clip,keepaspectratio]{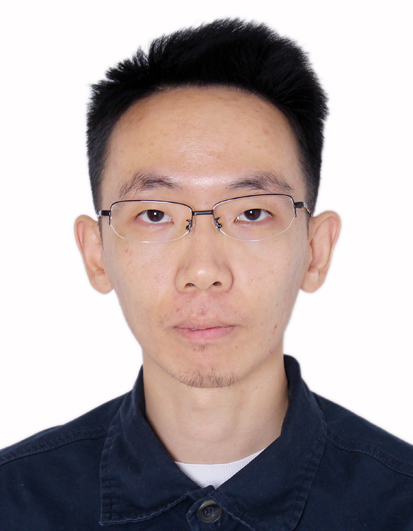}}]{Bin Zhu} was born in Changshu, Jiangsu Province, China in 1991. He received the B.Eng.~degree from Xi'an Jiaotong University, Xi'an, China in 2012 and the M.Eng.~degree from Shanghai Jiao Tong University, Shanghai, China in 2015, both in control science and engineering. In 2019, he obtained a Ph.D. degree in information engineering from University of Padova, Padova, Italy, and now he is a postdoctoral researcher in the same university.

His current research interest includes spectral estimation, rational covariance extension, and ARMA modeling.

\end{IEEEbiography}

\end{document}